\documentclass[11pt,a4paper]{article}
\usepackage{xcolor}
\usepackage{amsmath, amsthm, amssymb, amsfonts}
\usepackage{mathtools}
\usepackage{geometry}
\usepackage{hyperref}
\usepackage{enumitem}
\usepackage{bm}
\usepackage[T1]{fontenc}

\hypersetup{
  colorlinks=true,
  linkcolor=blue!60!black,
  citecolor=blue!60!black,
  urlcolor=blue!60!black
}

\theoremstyle{plain}
\newtheorem{theorem}{Theorem}[section]
\newtheorem{proposition}[theorem]{Proposition}
\newtheorem{property}[theorem]{Property}
\newtheorem{corollary}[theorem]{Corollary}
\newtheorem{lemma}[theorem]{Lemma}

\theoremstyle{definition}
\newtheorem{definition}[theorem]{Definition}

\theoremstyle{remark}
\newtheorem{remark}[theorem]{Remark}

\usepackage{authblk}
\usepackage{indentfirst}
\usepackage{comment}

\title{\textbf{ Factorization of  Isomorphisms of $(H,\theta)$-twisted Lie algebroids.}}
\author[1]{Nasser Saipele nansidi}
\author[2,3]{Bertuel Tangue Ndawa}
\author[4]{Luc Emery Diekouam Fotso }
\author[5]{Emmanuel Fouotsa }
\author[1]{Joseph Dongho}
\affil[1]{Faculty of science, University of Maroua, Cameroon.}
\affil[2]{University Institute of Technology, University of Ngaoundere, Cameroon.}
\affil[3]{Institut des Hautes \'{e}tudes Scientifiques, Universit\'{e} Paris-Saclay, France.}
\affil[4]{Higher Teacher Training College, University of Maroua, Cameroon.}
\affil[5]{ Higher Teacher Training College of Bambili, University of Bamenda, Cameroon.}

\begin{document}
\maketitle
\vspace*{1cm}
\begin{center}
In memory of Professor Joseph Dongho, who was present at the beginning of this work.
\end{center}
\vspace*{1cm}
\begin{abstract}
We study the isomorphism groupoid $\mathcal{T}(M)$ of $\theta$-almost twisted Poisson
($\theta$-atP) structures on a smooth manifold $M$, focusing on the internal
structure of its morphisms. A morphism in $\mathcal{T}(M)$ is a $C^\infty(M)$-linear
isomorphism $\Phi:\Omega^1(M)\to\Omega^1(M)$ that simultaneously intertwines the
anchor maps and the $(H,\theta)$-twisted Koszul brackets associated with two
$\theta$-atP structures. Every such morphism induces a canonical
isomorphism in $\theta$-atP cohomology.

We define a classifying functor
$$
\Delta : \mathrm{Mor}(\mathcal{T}(M)) \longrightarrow (Z^1_{\mathrm{dR}}(M)	,+),
\qquad \Delta(\Phi)=\theta' - \theta,
$$
which is additive under composition and partitions the morphisms into two
complementary families: the sub-groupoid $\mathcal{T}_{\mathrm{fix}}=\ker\Delta$ of isomorphisms
preserving $\theta$, and the family $\mathcal{T}_{\mathrm{mod}}$ of isomorphisms shifting $\theta$.
We describe each element in this partition.

\end{abstract}

\section{Introduction.}

Lie algebroids  arise across differential geometry, simultaneously generalizing  Lie algebras, tangent bundles, and foliations. The Lie algebroids of a Poisson manifold is the classical instance of this construction. Given a Poisson bivector $\Lambda$ on a
smooth manifold $M$, the cotangent bundle $T^*M$ carries a Lie algebroid
structure via the Koszul bracket
$
[\ ,\ ]_K:\Omega^1(M)\times\Omega^1(M)\to\Omega^1(M),\quad
[\alpha,\beta]_K=\mathcal{L}_{\Lambda^\#(\alpha)}(\beta)
-\mathcal{L}_{\Lambda^\#(\beta)}(\alpha)
-d\Lambda(\alpha,\beta),
$
and the anchor $\Lambda^\#:\Omega^{1}(M)\to \mathfrak{X}^{1}(M)$. The Chevalley--Eilenberg
cohomology of this algebroid reproduces Poisson cohomology in the sense
of Lichnerowicz~\cite{Lichnerowicz1977} and Huebschmann~\cite{Huebschmann1990}.

Motivated by topological sigma models and string
geometry, \v{S}evera and Weinstein~\cite{SeveraWeinstein2001} introduce on a manifold $M$, a  generalization of Poisson structure   under the name of twisted Poisson structure. This is  a bivector field $\Lambda$ together with   a closed 3-form $H\in\Omega^{3}(M)$
satisfying $
\frac{1}{2}[\Lambda,\Lambda]=\Lambda^{\#}(H).
$
One obtains a  (H,0)-twisted  Lie algebroid. This is a Lie algebroid  on $M$ induced by the $H$-deformation of  Koszul bracket.
The geometric quantization of these structures was studied by
Petalidou~\cite{Petalidou2007}, and their integrability theory by
Crainic--Fernandes~\cite{CrainicFernandes2003}.

A $\theta$-almost twisted Poisson manifold introduced in~\cite{NansidiTangueDongho2025b}
is inspired by sigma models with non-closed background 3-form fluxes~\cite{Chatzistavrakidis2020}.
More precisely, a $\theta$-almost twisted Poisson ($\theta$-atP) structure on a smooth manifold $M$
is a triple $(\Lambda,H,\theta)$, where $\Lambda\in \mathfrak{X}^{2}(M)$ is a bivector field,
and $H\in\Omega^{3}(M)$ and $\theta\in\Omega^{1}(M)$ are a 3-form and 1-form, respectively,
satisfying the equations
\begin{equation*}
	d\theta=0,\qquad
	\Lambda^{\#}(\theta)=0,\qquad
	\tfrac{1}{2}[\Lambda,\Lambda]=\Lambda^{\#}(H),\qquad
	dH=\theta\wedge H.
\end{equation*}
Each such triple defines an $(H,\theta)$-twisted Lie algebroid
$$
A(\Lambda,H,\theta)=\bigl(T^*M,[\,\cdot\,,\,\cdot\,]_{H,\theta},\Lambda^{\#}\bigr),
$$
where $[\,\cdot\,,\,\cdot\,]_{H,\theta}$ is the $(H,\theta)$-deformation of the
Koszul bracket introduced in~\cite{NansidiTangueDongho2025b} (see Section~\ref{sec:tools}).
This class simultaneously generalizes twisted Poisson algebroids (recovered when $\theta=0$),
and ordinary Poisson algebroids (recovered when $\theta=0$ and $H=0$).
The $\theta$-atP cohomology $H^{\ast}_{\theta\text{-atP}}(M,\Lambda,H,\theta)$,
constructed in~\cite{NansidiTangueDongho2025b} as the Chevalley--Eilenberg cohomology of
$A(\Lambda,H,\theta)$, is the corresponding cohomology theory.

This paper is inspired by~\cite{TNB2},  and addresses the following question:
what are the isomorphisms between two $(H,\theta)$-twisted Lie algebroids on $M$,
how do they form a groupoid, and how can they be factorized?
More precisely, given two $\theta$-atP structures $(\Lambda,H,\theta)$,  and
$(\Lambda',H',\theta')$ on $M$, we aim to characterize all $C^{\infty}(M)$-linear
isomorphisms $\Phi:\Omega^{1}(M)\to\Omega^{1}(M)$ satisfying both the anchor
condition $
(\Lambda')^{\#}\circ\Phi=\Lambda^{\#}
$
and the bracket condition
$$
\Phi\bigl([\alpha,\beta]_{H,\theta}\bigr)
=
[\Phi\alpha,\Phi\beta]_{H',\theta'},
\qquad \forall\,\alpha,\beta\in\Omega^{1}(M).
$$

\subsection{Main results.}
We denote by  $\mathcal{S}(M)$ the set of all $\theta$-atP structure on $M$. 
The isomorphisms of $(H,\theta)$-twisted Lie algebroids form a groupoid  $\mathcal{T}(M)$  over $\mathcal{S}(M)$, called the isomorphism groupoid of $(H,\theta)$-twisted Lie algebroids on $M$ (Proposition~\ref{thm:groupoid}).
 The structural
analysis of $\mathcal{T}(M)$ is organized around the classifying functor
$$
\Delta:\mathrm{Mor}\!\left(\mathcal{T}(M)\right)
\longrightarrow\bigl(Z^{1}_{\mathrm{dR}}(M),+\bigr),
\qquad \Delta(\Phi)=\theta'-\theta.
$$

The map $\Delta$ is additive under composition and sign-reversing under inversion
(Property~\ref{prop:Delta-functeur}).  Every $\Phi\in\mathrm{Mor}\!\left(\mathcal{T}(M)\right)$
preserves the characteristic sub-bundle $\ker\Lambda^{\#}\subset\Omega^{1}(M)$;
in particular, two isomorphic $(H,\theta)$-twisted Lie algebroids necessarily share the
same characteristic distribution (Proposition~\ref{preservation_ancre}).

The functor $\Delta$ partitions $\mathrm{Mor}\!\left(\mathcal{T}(M)\right)$ into two complementary
families. The sub-groupoid
$\mathcal{T}_{\mathrm{fix}}=\ker\Delta$ (Section~\ref{sec:tfix}) consists of
isomorphisms preserving $\theta$.
The gauge transformations $\Phi_{B}=\mathrm{id}+B^{\flat}\circ\Lambda^{\#}$, for
$B\in\Omega^{2}(M)$ rendering $\Phi_{B}$ invertible, form a sub-groupoid
$\mathrm{Gau}(M)\subset \mathcal{T}_{\mathrm{fix}}$ in which composition corresponds
to addition of two-forms (Propositions~\ref{prop:preserve}--\ref{prop:gauge}).
Our main result on $\mathcal{T}_{\mathrm{fix}}$ (Theorem~\ref{thm:classification}) states that every
$\Phi\in\mathcal{T}_{\mathrm{fix}}$ admits a factorization
\begin{equation*}
	\Phi=\Psi\circ\Phi_{B},
\end{equation*}
where $\Phi_{B}\in\mathrm{Gau}(M)$ and $\Psi$ is a $C^{\infty}(M)$-linear
automorphism of $\Omega^{1}(M)$ satisfying
$\Psi(\alpha)-\alpha\in\ker\Lambda^{\#}$ for all $\alpha\in\Omega^{1}(M)$.
The residual factor $\Psi$ thus acts trivially on the quotient
$\Omega^{1}(M)/\ker\Lambda^{\#}$. In particular, when $\Lambda$ is nondegenerate,
$\ker\Lambda^{\#}=0$ and $\mathcal{T}_{\mathrm{fix}}$ coincides with $\mathrm{Gau}(M)$.
The family $\mathcal{T}_{\mathrm{mod}}$ (Section~\ref{sec:Tmod}) consists of isomorphisms
shifting $\theta$.
The canonical examples of morphisms in $\mathcal{T}_{\mathrm{mod}}$   are the conformal Casimir transformations
$\Psi_{f}(\alpha)=e^{f}\alpha$, defined for Casimir functions
$$
f\in\mathrm{Cas}(\Lambda)
=\{f\in C^{\infty}(M):\Lambda^{\#}(df)=0\},
$$
which shift $\theta$ by the exact form $df$ (Proposition~\ref{thm:main}).
We show that every $\Phi\in\Delta^{-1}(df)$ factors as
$\Phi=\Psi\circ\Theta_{f,B}$, where
$
\Theta_{f,B}(\alpha)=e^{f}\alpha+i_{\Lambda^{\#}\!\alpha}\,B,
$
and $\Psi$ is a residual automorphism trivial on $\Omega^{1}(M)/\ker\Lambda^{\#}$
(Theorem~\ref{thm:structure-fibre}).
Finally, we establish that the conformal-gauge sub-groupoid $\mathcal{T}_{\mathrm{cg}}$,
generated under composition by all gauge transformations and all conformal Casimir transformations,
satisfies
$\Delta\!\left(\mathrm{Mor}(\mathcal{T}_{\mathrm{cg}})\right)\subseteq B^{1}_{\mathrm{dR}}(M)$.
Consequently, no morphism in $\mathcal{T}_{\mathrm{cg}}$ can realise a shift whose de~Rham
class is non-trivial:
$$
\Delta^{-1}(\eta)\cap\mathrm{Mor}(\mathcal{T}_{\mathrm{cg}})=\emptyset
\qquad\text{whenever}\quad
[\eta]\neq 0\;\text{ in }\;H^{1}_{\mathrm{dR}}(M,\mathbb{R}).
$$
(see Proposition~\ref{prop:nogo}).
However, the existence of isomorphisms connecting $\theta$-atP structures whose backgrounds differ
by a non-exact closed form is proved.
The characterization of the latter remains open.

The paper is organized as follows. Section~\ref{sec:tools} recalls the background:
the Schouten--Nijenhuis bracket, $\theta$-twisted Courant algebroids, the Dirac
characterization of $\theta$-atP structures via the graph construction, and
the definition of $\theta$-atP cohomology.
Section~\ref{sec:groupoid} defines  $\mathcal{T}(M)$, and 
establishes the classifying functor $\Delta$.
Section~\ref{sec:tfix} is devoted to  $\mathcal{T}_{\mathrm{fix}}$ and its decomposition.
Section~\ref{sec:Tmod} treats  $\mathcal{T}_{\mathrm{mod}}$.

\section{Tools}
\label{sec:tools}
This section presents some results that will be needed in the sequel.

\subsection{Notations and conventions.}
Throughout, all manifolds and maps are smooth. Let $M$ be a  manifold.
The algebra of smooth functions on $M$ is denoted by $C^{\infty}(M)$.

For $k\geq 0$, $\Omega^{k}(M)$ and $\mathfrak{X}^{k}(M)$ denote the spaces of smooth
$k$-forms and $k$-vector fields, respectively. We set
$\Omega(M)=\bigoplus_{k\geq 0}\Omega^{k}(M)$ and $\mathfrak{X}(M)=\bigoplus_{k\geq 0}\mathfrak{X}^{k}(M)$.
The tangent bundle and the cotangent bundle of $M$ are denoted by $TM$ and $T^{*}M$,
respectively.

Let $X\in \mathfrak{X}(M)$ and let $\psi\in\Omega^{k}(M)$.
We denote by $i_{X}\psi$ the interior product (contraction) of $\psi$, defined by
\begin{align*}
	(i_{X}\psi)(X_{1},\dots,X_{k-1})
	&=\psi(X,X_{1},\dots,X_{k-1})
\end{align*}
for every $X_{1},\dots,X_{k-1}\in\mathfrak{X}^{1}(M)$.
Observe that the map $i_{X}:\Omega(M)\to\Omega^{k-1}(M)$ is a
$C^{\infty}(M)$-linear endomorphism of $\Omega(M)$.

Let $\Lambda\in \mathfrak{X}^{2}(M)$ be a bivector field. It induces naturally a morphism of vector bundles
$\Lambda^\#: T^{*}M\longrightarrow TM$, and, more generally, a morphism of $C^\infty(M)$-modules
\[
\Lambda^\#:\Omega^{k}(M)\longrightarrow \mathfrak{X}^{k}(M),
\]
defined by
\begin{align*}
	\Lambda^\#(\psi)(\alpha_{1},\dots,\alpha_{k})
	&= (-1)^{k}\psi\bigl(\Lambda^\#(\alpha_{1}),\dots,\Lambda^\#(\alpha_{k})\bigr),
\end{align*}
for every $\psi\in \Omega^{k}(M)$ and $\alpha_{1},\dots,\alpha_{k}\in \Omega^{1}(M)$.
In particular, $\Lambda^\#(f)=f$ for $f\in C^\infty(M)$, and $\Lambda^\#(\alpha)$ is a vector field for every
$\alpha\in \Omega^{1}(M)$ given by
\[
\Lambda^\#(\alpha)(\beta)=\Lambda(\alpha,\beta),\qquad \forall\,\beta\in\Omega^{1}(M).
\]

A $2$-form $B\in\Omega^{2}(M)$ induces also a morphism
$B^\flat:\mathfrak{X}^{1}(M)\to \Omega^{1}(M)$ defined by
\[
B^\flat(X)=i_{X}B,\qquad \forall\,X\in \mathfrak{X}^{1}(M).
\]

The Koszul bracket $[\ ,\ ]_{K}$ associated to a bivector field $\Lambda$ is defined by
\[
[\alpha,\beta]_{K}
=\mathcal{L}_{\Lambda^\#(\alpha)}(\beta)
-\mathcal{L}_{\Lambda^\#(\beta)}(\alpha)
-d\Lambda(\alpha,\beta),
\qquad \alpha,\beta\in \Omega^{1}(M).
\]

We denote by $[P,Q]$ the Schouten--Nijenhuis bracket of multivector fields
$P,Q\in\Omega(M)$. See \cite[Theorem 2.8]{CrainicFernandesMarcut2021}
for properties of this bracket.

All bivector fields $\Lambda$ are assumed to have constant rank. This ensures that $D=Im\Lambda^\#$ is a smooth subbundle of $TM$, as required in section~\ref{sec:tfix} and \ref{sec:Tmod}.

\subsection{$\theta$-almost twisted Poisson structures and their Dirac characterization.}
\label{sec:courant_algebroid}
Let $\theta$ be a closed $1$-form on $M$.

The Lichnerowicz--de~Rham differential (see \cite[Definition 2.2]{Chantraine}) is defined by
\begin{equation*}
	d_{\theta}:\Omega(M)\longrightarrow \Omega(M),\qquad
	d_{\theta}(\alpha)=d\alpha-\theta\wedge\alpha .
\end{equation*}
Since $d\theta=0$, we have $d_{\theta}^{2}=0$.

For $X\in \mathfrak{X}^{1}(M)$, the modified Lie derivative is given by
\begin{equation*}
	\mathcal{L}^{\theta}_{X}=d_{\theta}\, i_{X}+ i_{X} d_{\theta}
	=\mathcal{L}_{X}-\theta(X).
\end{equation*}
Moreover, it satisfies the usual Cartan identities:
\begin{align}
	[\mathcal{L}^{\theta}_{X},d_{\theta}] &= 0, \label{cartan1}\\
	[\mathcal{L}^{\theta}_{X}, i_{Y}] &= i_{[X,Y]}, \label{cartan2}\\
	[\mathcal{L}^{\theta}_{X},\mathcal{L}^{\theta}_{Y}] &= \mathcal{L}^{\theta}_{[X,Y]}. \label{cartan3}
\end{align}

 We consider the vector bundle $E = TM \oplus T^*M$ endowed with the anchor
 $\rho: E\to TM$, $\rho(X,\xi)=X$,
 the nondegenerate symmetric bilinear pairing
 $	\langle (X,\xi),(Y,\eta) \rangle = i_X\eta + i_Y\xi,$

 the $\theta$-twisted Dorfman bracket defined for $e_i=(X_i,\xi_i)\in\Gamma(E)$, $i=1,2$, by
\begin{equation*}
	e_1 \circ_\theta e_2 =( [X_1,X_2]\ ,\ \mathcal{L}^\theta_{X_1} \xi_2 - i_{X_2} d_{\theta} \xi_1),
\end{equation*}
 and the $\theta$-twisted differential $D_{\theta}:C^\infty(M)\to\Gamma(E)$,
$D_{\theta}f=(0,\ d_{\theta}f).$
 We obtain the following properties.

\begin{proposition}\label{prop:Courant-theta}
 For all $e_i\in\Gamma(E)$ and $f\in C^\infty(M)$, the quadruple $(E,\rho,\langle\cdot,\cdot\rangle,\circ_\theta,D_\theta)$ satisfies de the following properties
\begin{enumerate}
\item $\rho(e_1\circ_\theta e_2) = [\rho(e_1),\rho(e_2)]$;
\item $e_1\circ_\theta e_2 + e_2\circ_\theta e_1 = D_\theta\langle e_1,e_2\rangle$;
\item $\langle e_1\circ_\theta e_2, e_3\rangle + \langle e_2, e_1\circ_\theta e_3\rangle = \mathcal{L}^\theta_{\rho(e_1)} \langle e_2,e_3\rangle$;
\item $e_1\circ_\theta (e_2\circ_\theta e_3) = (e_1\circ_\theta e_2)\circ_\theta e_3 + e_2\circ_\theta (e_1\circ_\theta e_3)$;
\item $D_\theta f \circ_\theta e_1 = 0$,\quad $e_1 \circ_\theta D_\theta f = D_\theta\langle e_1, D_\theta f\rangle$.
\end{enumerate}
We will term by \emph{$\theta$-twisted Courant algebroid } the quadruple $(E,\rho,\langle\cdot,\cdot\rangle,\circ_\theta,D_\theta)$ satisfying such properties.
\end{proposition}

\begin{proof}
\textit{Axiom 1:} The vector component of $e_1\circ_\theta e_2$ is $[X_1,X_2]$, which is exactly equal to $[\rho(e_1),\rho(e_2)]$.

\smallskip
\textit{Axiom 2:} Using $\mathcal{L}^\theta_X = d_{\theta} i_X + i_X d_{\theta}$, the $1$-form part of
$e_1\circ_\theta e_2 + e_2\circ_\theta e_1$ equals
\begin{align*}
	\bigl(d_{\theta} i_{X_1}\xi_2 + i_{X_1}d_{\theta}\xi_2 - i_{X_2}d_{\theta}\xi_1\bigr)
	&+ \bigl(d_{\theta} i_{X_2}\xi_1 + i_{X_2}d_{\theta}\xi_1 - i_{X_1}d_{\theta}\xi_2\bigr)\\
	&= d_{\theta}\bigl(i_{X_1}\xi_2 + i_{X_2}\xi_1\bigr)\\
	&= d_{\theta}\langle e_1,e_2\rangle.
\end{align*}
The vector parts cancel since $[X_2,X_1]=-[X_1,X_2]$, hence
$e_1\circ_\theta e_2 + e_2\circ_\theta e_1 = D_\theta\langle e_1,e_2\rangle$.

\smallskip
\textit{Axiom 3:} Expanding the left-hand side and using $i_{X_3}i_{X_2}=-i_{X_2}i_{X_3}$, one obtains
$$
i_{[X_1,X_2]}\xi_3 + i_{X_3}\mathcal{L}^\theta_{X_1}\xi_2
+ i_{X_2}\mathcal{L}^\theta_{X_1}\xi_3 + i_{[X_1,X_3]}\xi_2.
$$
On the other hand, applying $\mathcal{L}^\theta_{X_1}$ to
$\langle e_2,e_3\rangle = i_{X_2}\xi_3 + i_{X_3}\xi_2$, and using \eqref{cartan2},
we obtain exactly the same expression.

\smallskip
\textit{Axiom 4:} Set $\xi_{23}:=\mathcal{L}^\theta_{X_2}\xi_3 - i_{X_3}d_{\theta}\xi_2$, so that
$e_2\circ_\theta e_3 = \bigl([X_2,X_3],\ \xi_{23}\bigr)$.
Then
\begin{equation*}
	e_1\circ_\theta (e_2\circ_\theta e_3)
	=
	\Bigl([X_1,[X_2,X_3]],\ \mathcal{L}^\theta_{X_1}\mathcal{L}^\theta_{X_2}\xi_3
	- \mathcal{L}^\theta_{X_1}i_{X_3}d_{\theta}\xi_2
	- i_{[X_2,X_3]}d_{\theta}\xi_1\Bigr).
\end{equation*}
For the right-hand side, one computes $(e_1\circ_\theta e_2)\circ_\theta e_3$ and
$e_2\circ_\theta (e_1\circ_\theta e_3)$ separately.
The vector parts give
$[[X_1,X_2],X_3] + [X_2,[X_1,X_3]] = [X_1,[X_2,X_3]]$ by the Jacobi identity.
For the $1$-form parts, identity \eqref{cartan3} implies
$\mathcal{L}^\theta_{X_1}\mathcal{L}^\theta_{X_2}
= \mathcal{L}^\theta_{X_2}\mathcal{L}^\theta_{X_1}+\mathcal{L}^\theta_{[X_1,X_2]}$,
which matches the $\xi_3$-terms.
Applying \eqref{cartan2} to rewrite
$\mathcal{L}^\theta_{X_1}i_{X_3}
= i_{X_3}\mathcal{L}^\theta_{X_1}+i_{[X_1,X_3]}$,
and using the commutation $d_{\theta}\mathcal{L}^\theta_{X_1}=\mathcal{L}^\theta_{X_1}d_{\theta}$,
which follows from \eqref{cartan1},
all remaining terms cancel identically.
Hence Axiom 4 holds.

\smallskip
\textit{Axiom 5:} Since $D_\theta f=(0,d_{\theta} f)$ has vanishing vector part, we have
$D_\theta f\circ_\theta e=0$.
For the second identity, we compute
$$
e\circ_\theta D_\theta f
= \mathcal{L}^\theta_X d_{\theta} f
= d_{\theta}\bigl(i_X d_{\theta} f\bigr)
= D_\theta\bigl(i_X d_{\theta} f\bigr)
= D_\theta\langle e,D_\theta f\rangle.
$$
This completes the proof.
\end{proof}

Let $H\in\Omega^3(M)$. We define a modified bracket by
\begin{equation*}
	e_1 \circ_{\theta,H} e_2
	=
	e_1 \circ_\theta e_2 + \bigl(0,\, i_{X_2} i_{X_1} H\bigr).
\end{equation*}

\begin{proposition}
	The bracket $\circ_{\theta,H}$ makes $(E,\rho,\langle\cdot,\cdot\rangle,D_\theta)$ a $\theta$-Courant algebroid
	if and only if
	\begin{equation}\label{closedH}
		d_{\theta} H = 0.
	\end{equation}
	When~\eqref{closedH} holds, we call $(E,\rho,\langle\cdot,\cdot\rangle,\circ_{\theta,H},D_\theta)$ an
	\emph{$H$-twisted $\theta$-Courant algebroid}.
\end{proposition}

\begin{proof}
	We set $h(e_1,e_2):=\bigl(0, i_{X_2} i_{X_1} H\bigr)$. Then
	$e_1\circ_{\theta,H}e_2 = e_1\circ_\theta e_2 + h(e_1,e_2)$.
	Since $h$ is antisymmetric in $X_1,X_2$ and has no vector component, Axioms 1, 2, 3, and 5
	hold unchanged.
	
	It remains to verify Axiom 4 (the Leibniz rule). A direct expansion gives
	$$
	e_1\circ_{\theta,H}(e_2\circ_{\theta,H}e_3)
	=
	e_1\circ_\theta(e_2\circ_\theta e_3)
	+
	\Bigl(0,\,\mathcal{L}^\theta_{X_1} i_{X_3} i_{X_2} H + i_{[X_2,X_3]} i_{X_1} H\Bigr).
	$$
	Similarly, after collecting all terms in
	$(e_1\circ_{\theta,H}e_2)\circ_{\theta,H}e_3 + e_2\circ_{\theta,H}(e_1\circ_{\theta,H}e_3)$
	and using Axiom 4 for $\circ_\theta$, the discrepancy between the two sides is
	$(0,R)$, where
	\begin{align*}
		R
		&=
		\mathcal{L}^\theta_{X_1} i_{X_3}i_{X_2}H + i_{[X_2,X_3]} i_{X_1} H \\
		&\quad
		-\Bigl(
		-i_{X_3}d_{\theta} i_{X_2}i_{X_1}H
		+ i_{X_3} i_{[X_1,X_2]} H
		+ \mathcal{L}^\theta_{X_2} i_{X_3}i_{X_1}H
		+ i_{[X_1,X_3]} i_{X_2} H
		\Bigr).
	\end{align*}
	
	Let $X_0$ be any vector field on $M$. According to
	$i_{X_0}\mathcal{L}^\theta_X = \mathcal{L}^\theta_X i_{X_0} - i_{[X,X_0]}$
	and
	$i_{X_0}d_{\theta} = \mathcal{L}^\theta_{X_0}-d_{\theta} i_{X_0}$,
	together with the antisymmetry of interior products, we obtain
	\begin{align*}
		i_{X_0}R &= -(d_{\theta} H)(X_0,X_1,X_2,X_3).
	\end{align*}
	Hence, $R=0$ if and only if $d_{\theta} H=0$. This proves the Proposition.
\end{proof}

\begin{definition}
	A Dirac structure in $(E,\circ_{\theta,H})$ is a sub-bundle $L\subset E$ such that
	\begin{enumerate}
		\item[(i)]
		$L$ is isotropic and of maximal rank: $\langle L,L\rangle=0$ and $\mathrm{rank}\,L=\tfrac{1}{2}\mathrm{rank}\,E$;
		\item[(ii)]
		$\Gamma(L)$ is closed under $\circ_{\theta,H}$.
	\end{enumerate}
\end{definition}

Given a bivector field $\Lambda\in\mathfrak{X}^1(M)$, we define the graph of $\Lambda$ by
$$
L_\Lambda
=
\bigl\{(\Lambda^\#\xi,\xi)\ \big|\ \xi\in T^*M\bigr\}
\subset E.
$$
The sub-bundle $L_\Lambda$ is always isotropic and of maximal rank, since
\begin{align*}
	\langle (\Lambda^\#\xi,\xi),(\Lambda^\#\eta,\eta) \rangle
	&=
	\Lambda(\xi,\eta)+\Lambda(\eta,\xi)
	=
	0.
\end{align*}

\begin{proposition}\label{dirac}
	Let $\theta$ be a closed $1$-form and $H$ a $3$-form such that $d_{\theta}H=0$.
	Assume that $\Lambda^\#(\theta)=0$.
	Then $L_\Lambda$ is a Dirac structure for $\circ_{\theta,H}$ if and only if
	\begin{equation}\label{twistedPoisson}
		\tfrac{1}{2}[\Lambda,\Lambda] = \Lambda^\#(H).
	\end{equation}
\end{proposition}

\begin{proof}
Let $\xi,\eta\in\Omega^1(M)$. We set $X=\Lambda^\#\xi$, $Y=\Lambda^\#\eta$. The
condition that $\Gamma(L_\Lambda)$ is closed under $\circ_{\theta,H}$ reads
\begin{equation}\label{diracCond}
  [X,Y] = \Lambda^\#\!\bigl(\mathcal{L}^\theta_X\eta - i_Yd_{\theta}\xi + i_Yi_XH\bigr).
\end{equation}
Using $\mathcal{L}^\theta_X = L_X - \theta(X)$ and
$d_{\theta}\xi = d\xi - \theta\wedge\xi$,  the right-hand side of \eqref{diracCond} gives
$$
  \Lambda^\#\!\bigl(\mathcal{L}^\theta_X\eta - i_Yd_{\theta}\xi + i_Yi_XH\bigr)
  =\Lambda^\#\!\bigl( L_X\eta - i_Y\, d\xi\bigr) - \Lambda^\#\!\bigl(\theta(X)\eta + \theta(Y)\xi
    - \xi(Y)\theta + i_Yi_XH\bigr).
$$
The classical Schouten--Nijenhuis identity (see, e.g.,~\cite[\S 2.2]{CrainicFernandesMarcut2021}) gives
\begin{equation}\label{SN}
  [X,Y] - \Lambda^\#(L_X\eta - i_Y\,d\xi)
  = \tfrac{1}{2}[\Lambda,\Lambda](\xi,\eta, ).
\end{equation}
Therefore, condition~\eqref{diracCond} is equivalent to, for every $1$-form $\zeta$,
\begin{align*}
  \tfrac{1}{2}[\Lambda,\Lambda](\xi,\eta,\zeta)
  &= -H(X,Y,\Lambda^\#\zeta)
    - \theta(X)\Lambda(\eta,\zeta)
    + \theta(Y)\Lambda(\xi,\zeta)\\
    &- \Lambda(\xi,\eta)\,\Lambda(\theta,\zeta).
\end{align*}

 Using the fact that $\Lambda^\#(\theta)=0$, we get

$$
  \tfrac{1}{2}[\Lambda,\Lambda](\xi,\eta,\zeta)
  = \Lambda^\#(H)(\xi,\eta,\zeta).
$$
 Since this must hold for all $\xi,\eta,\zeta$, it is equivalent to~\eqref{twistedPoisson}.

\end{proof}

\begin{definition}[\cite{NansidiTangueDongho2025b}]\label{def_theta_alP}
  A \emph{$\theta$-almost twisted Poisson ($\theta$-atP) manifold} is a manifold together with a triple
  $(\Lambda,H,\theta)$ where
  $\Lambda\in\mathfrak{X}^{2}(M)$,
  $H\in\Omega^{3}(M)$,
  $\theta\in\Omega^{1}(M)$, satisfying: $d\theta = 0$, $\Lambda^{\#}(\theta) = 0$,  $\tfrac{1}{2}[\Lambda,\Lambda] = \Lambda^{\#}(H)$, and $dH = \theta\wedge H$.
\end{definition}

\subsection{$\theta$-almost twisted Poisson cohomology}
\label{subsec:cohomology}
Let   $(\Lambda,H,\theta)$ be a  $\theta$-almost twisted Poisson structure  on $M$.
The module $\Omega^{1}(M)$ carries the \emph{$(H,\theta)$-twisted
Koszul bracket}
\begin{equation*}
  [\alpha,\beta]_{H,\theta}
  = [\alpha,\beta]_{K}
    + i_{\Lambda^{\#}(\beta)}\,i_{\Lambda^{\#}(\alpha)}\,H
    + \Lambda(\alpha,\beta)\cdot\theta,
\end{equation*}
and the triple
$\mathcal{A}(\Lambda,H,\theta)
 =(\Omega^{1}(M),[\cdot,\cdot]_{H,\theta},\Lambda^{\#})$
is a Lie algebroid over $M$ (see \cite{NansidiTangueDongho2025b}).\\

The Chevalley--Eilenberg complex of $\mathcal{A}(\Lambda,H,\theta)$
is $(\mathfrak{X}^{*}(M),\partial_{H,\theta})$, where
$\partial_{H,\theta}^{n}:\mathfrak{X}^{n}(M)\to\mathfrak{X}^{n+1}(M)$ is defined by
\begin{align}
  \partial_{H,\theta}^{n}(v)(\alpha_1,\ldots,\alpha_{n+1})
  &= \sum_{i=1}^{n+1}(-1)^{i-1}
     \Lambda^{\#}(\alpha_i)\bigl(v(\alpha_1,\ldots,\widehat{\alpha_i},\ldots,\alpha_{n+1})\bigr)
  \nonumber\\
  &\quad
  + \!\!\sum_{1\le i<j\le n+1}\!\!(-1)^{i+j}
    v\bigl([\alpha_i,\alpha_j]_{H,\theta},\alpha_1,
           \ldots,\widehat{\alpha_i},\ldots,\widehat{\alpha_j},\ldots,\alpha_{n+1}\bigr).
  \label{eq:opcobord}
\end{align}
The $\theta$-almost twisted Poisson cohomology is given as
$H^{n}_{\theta\text{-atP}}(M) = \dfrac{Ker\partial_{H,\theta}^{n}}{Img\partial_{H,\theta}^{n-1}}$.\\
The  identity $ \partial_{H,\theta}\circ\Lambda^{\#} = -\Lambda^{\#}\circ d$
yields a cohomology morphism\\
$$\Lambda^{\#}: H^{*}_{\mathrm{dR}}(M,\mathbb{R})\to H^{*}_{\theta\text{-atP}}(M,\Lambda,H,\theta),$$
which is an isomorphism when $\Lambda$ is nondegenerate.

\section{The groupoid of $(H,\theta)$-twisted Lie algebroids.}\label{sec:groupoid}
 We denote by $\mathcal{S}(M)$  the set of all $\theta$-atP structures on $M$. We first  specify the morphisms under consideration. 

\begin{definition}\label{def:morphism}
Let $(\Lambda, H, \theta)$ and $(\Lambda', H', \theta')$ be two $\theta$-atP
structures on $M$. An isomorphism of $(H,\theta)$-twisted Lie
algebroids from $A(\Lambda, H, \theta)$ to $A(\Lambda', H', \theta')$ is a
$C^\infty(M)$-linear isomorphism $\Phi : \Omega^1(M) \to \Omega^1(M)$ satisfying
\begin{equation}\label{eq:morphism-bis}
  (\Lambda')^\# \circ \Phi = \Lambda^\#
  \qquad\text{and}\qquad
  \Phi\bigl([\alpha,\beta]_{H,\theta}\bigr)
  = [\Phi\alpha,\Phi\beta]_{H',\theta'}
  \quad \forall\;\alpha,\beta \in \Omega^1(M).
\end{equation}
We call such  a $\Phi$  a morphism from $(\Lambda,H,\theta)$ to
$(\Lambda',H',\theta')$  in $\mathcal{S}(M)$, and use the notation
$\Phi : (\Lambda,H,\theta) \to (\Lambda',H',\theta')$.
\end{definition}

\begin{proposition}\label{thm:groupoid}
	The collection $\mathcal{T}(M)$, whose objects are the $\theta$-atP structures $\mathcal{S}(M)$,
	and whose morphisms from $(\Lambda,H,\theta)$ to $(\Lambda',H',\theta')$
	are the isomorphisms of $(H,\theta)$-twisted Lie algebroids in the sense of
	Definition~\ref{def:morphism},
	is a groupoid under composition of maps.
\end{proposition}

\begin{proof}
	We verify the groupoid axioms.\smallskip
	
	\emph{Composition:}
	Let
	$\Phi_1 : (\Lambda,H,\theta) \to (\Lambda',H',\theta')$
	and
	$\Phi_2 : (\Lambda',H',\theta') \to (\Lambda'',H'',\theta'')$
	be two morphisms.
	Using \eqref{eq:morphism-bis}, we obtain
\begin{align*}
   (\Lambda'')^\# \circ (\Phi_2 \circ \Phi_1)
   &= \bigl((\Lambda'')^\# \circ \Phi_2\bigr) \circ \Phi_1
   = (\Lambda')^\# \circ \Phi_1 = \Lambda^\#,
   \end{align*}
   and
   \begin{align*}
   (\Phi_2 \circ \Phi_1)\bigl([\alpha,\beta]_{H,\theta}\bigr)
   &= \Phi_2\bigl([\Phi_1\alpha, \Phi_1\beta]_{H',\theta'}\bigr)
   = [(\Phi_2\circ\Phi_1)\alpha,\, (\Phi_2\circ\Phi_1)\beta]_{H'',\theta''}.
\end{align*}

Hence $\Phi_2\circ\Phi_1$ satisfies both conditions of \eqref{eq:morphism-bis}
for the pair
$((\Lambda,H,\theta),(\Lambda'',H'',\theta''))$.

\emph{Identities:}
The identity map $id_{\Omega^1(M)}$ satisfies both conditions with
$(\Lambda',H',\theta')=(\Lambda,H,\theta)$.

\emph{Inverses:}
From $(\Lambda')^\#\circ\Phi=\Lambda^\#$, we obtain
$(\Lambda')^\#=\Lambda^\#\circ\Phi^{-1}$.
Setting $\alpha'=\Phi\alpha$ and $\beta'=\Phi\beta$ in the bracket condition yields
$\Phi^{-1}\bigl([\alpha',\beta']_{H',\theta'}\bigr)
=
[\Phi^{-1}\alpha',\Phi^{-1}\beta']_{H,\theta}$.
Therefore $\Phi^{-1}$ is a morphism from $(\Lambda',H',\theta')$ to $(\Lambda,H,\theta)$.
\end{proof}

\begin{corollary}\label{cor:cohom}
Every morphism $\Phi \in Mor(\mathcal{T}(M))$ induces a natural isomorphism
$$
   H^{*}_{\theta\text{-atP}}(M,\Lambda,H,\theta)
  \xrightarrow{\;\sim\;}
   H^{*}_{\theta\text{-atP}}(M,\Lambda',H',\theta').
$$
\end{corollary}

\begin{proof}
Extend $\Phi$ to multivectors by
$(\Phi P)(\alpha_1,\ldots,\alpha_k) := P(\Phi\alpha_1,\ldots,\Phi\alpha_k)$.
Using \eqref{eq:opcobord} and \eqref{eq:morphism-bis},
a direct computation yields
$$
\partial_{H,\theta}\circ\Phi = \Phi\circ\partial_{H',\theta'}.
$$
Since $\Phi$ is an isomorphism on $\Omega^1(M)$, the extended map is a chain
isomorphism and therefore descends to an isomorphism on cohomology.
\end{proof}
 A natural invariant of a morphism is the difference $\theta'-\theta$, which depends only on the source and target structures.  We denote by $Z^1_{\mathrm{dR}}(M)$ the set of closed  1-forms; that is
 $Z^1_{\mathrm{dR}}(M)=\{ \alpha\in\Omega^1(M) \ |\  d\alpha=0\} $

\begin{definition}\label{def:Delta}
	The \emph{classifying functor} of $\mathcal{T}(M)$ is the map
	$$
	\Delta :  \mathrm{Mor}\bigl(\mathcal{T}(M)\bigr)
	\longrightarrow Z^1_{\mathrm{dR}}(M), \qquad
	\Delta(\Phi) = \theta' - \theta,
	$$
	where $\Phi : (\Lambda, H, \theta) \to (\Lambda', H', \theta')$.
\end{definition}

The following result comes directly from Definition~\ref{def:Delta}.
\begin{property}\label{prop:Delta-functeur}
	The map $\Delta$ is a functor from $\mathcal{T}(M)$ to the abelian group
	$(Z^1_{\mathrm{dR}}(M), +)$, viewed as a groupoid with a single object.
	More precisely,
	\begin{enumerate}
		\item[(i)] $\Delta\bigl(\mathrm{Id}_{(\Lambda,H,\theta)}\bigr) = 0$;
		\item[(ii)] $\Delta(\Phi_2 \circ \Phi_1) = \Delta(\Phi_2) + \Delta(\Phi_1)$
	for any composable pair $\Phi_1,\Phi_2$;
		\item[(iii)] $\Delta(\Phi^{-1}) = -\Delta(\Phi)$.
	\end{enumerate}
\end{property}

\begin{remark}
	The map $\Delta$ depends only on the one-forms $\theta$ and $\theta'$,
	and not on $\Lambda$, $H$, or the map $\Phi$ itself. It is
	the most elementary invariant attached to a morphism.
	We set
	$$
	\mathcal{T}_{\mathrm{fix}} = \ker\Delta
	= \bigl\{\Phi \in \mathrm{Mor}(\mathcal{T}(M)) : \theta' = \theta\bigr\},
	\qquad
	\mathcal{T}_{\mathrm{mod}} = \mathrm{Mor}(\mathcal{T}(M)) \setminus \mathcal{T}_{\mathrm{fix}}.
	$$
	It follows that $\mathcal{T}_{\mathrm{fix}}=\ker\Delta$ is a sub-groupoid of
	$\mathcal{T}(M)$.
\end{remark}

Every element of $\mathrm{Mor}(\mathcal{T}(M))$ has the following property.
\begin{proposition}\label{preservation_ancre}
	Let $\Phi : (\Lambda,H,\theta) \to (\Lambda',H',\theta')$ be a morphism
	in $\mathcal{T}(M)$. Then $\Phi$ preserves the characteristic sub-bundle
	$\ker\Lambda^\#$; more precisely,
	$$
	\Phi\bigl(\ker\Lambda^\#\bigr) = \ker(\Lambda')^\# = \ker\Lambda^\#.
	$$
	In other words, any two isomorphic $(H,\theta)$-twisted Lie algebroids
	share the same characteristic sub-bundle.
\end{proposition}

\begin{proof}
	Set $K = \ker\Lambda^\#$ and $D = \mathrm{Im}\,\Lambda^\#$.
	From the anchor condition $(\Lambda')^\#\circ\Phi = \Lambda^\#$, and the
	surjectivity of $\Phi$, we obtain
	$\mathrm{Im}\,(\Lambda')^\# = \Lambda^\#\bigl(\Omega^1(M)\bigr) = D$.
	Since, for any bivector $\Pi$, one has
	$\ker\Pi^\# = (\mathrm{Im}\,\Pi^\#)^\circ
	:= \{\alpha\in\Omega^1(M)\mid \alpha(X)=0,\ \forall X\in \mathrm{im}\,\Pi^\#\}$,
	it follows that
	$\ker(\Lambda')^\# = D^\circ = K$.
	
	For any $\kappa \in K$, the anchor condition gives
	$(\Lambda')^\#(\Phi\kappa) = \Lambda^\#\kappa = 0$, hence
	$\Phi\kappa \in \ker(\Lambda')^\# = K$, i.e.\ $\Phi(K)\subseteq K$.
	Applying the same argument to $\Phi^{-1}$ yields $\Phi^{-1}(K)\subseteq K$,
	and bijectivity implies $\Phi(K)=K$.
\end{proof}

\section{Isomorphisms fixing $\theta$: the sub-groupoid $\mathcal{T}_{\mathrm{fix}}$}
\label{sec:tfix}
 Throughout this section, $\Phi : (\Lambda,H,\theta) \to (\Lambda',H',\theta')$ denotes a morphism with $\theta'=\theta$. We identify the sub-groupoid $\mathrm{Gau}(M)\subset\mathcal{T}_{\mathrm{fix}}$ of gauge transformations and show that every element of $\mathcal{T}_{\mathrm{fix}}$ factors through one.

Let $(E,\rho,\langle\cdot,\cdot\rangle,\circ_\theta,D_\theta)$ be a $\theta$-Courant algebroid as defined in
Section~\ref{sec:courant_algebroid}.
For any $2$-form $B\in\Omega^2(M)$, we define the vector bundle automorphism
\begin{equation*}
	e^{B}\colon E\longrightarrow E,\qquad
	e^{B}(X,\xi) = \bigl(X,\ \xi + i_X B\bigr).
\end{equation*}
A direct verification shows that $e^{B}$ preserves both the anchor and the symmetric pairing.

\begin{proposition}\label{prop:Bgauge}
For any $e_1,e_2\in\Gamma(E)$,
\begin{equation}\label{eq:Bformula}
  e^B(e_1)\circ_{\theta,H} e^B(e_2)
  = e^B\bigl(e_1\circ_{\theta,H} e_2\bigr)
    + \bigl(0,\,i_{X_2}i_{X_1}d_{\theta} B\bigr).
\end{equation}
In particular, $e^B$ is an isomorphism of $\theta$-Courant algebroids from
$(E,\circ_{\theta,H})$ to $(E,\circ_{\theta,H'})$ with $H' = H - d_{\theta} B$.
\end{proposition}

\begin{proof}
We set $\xi_i'=\xi_i+i_{X_i}B$ for $i=1,2$, so that $e^{B}(e_i)=(X_i\;,\;\xi_i')$. Then
$$
  e^B(e_1)\circ_{\theta,H}e^B(e_2)
  = \Bigl([X_1,X_2]\  ,\   \mathcal{L}^\theta_{X_1}\xi'_2 - i_{X_2}d_{\theta}\xi'_1 + i_{X_2}i_{X_1}H\Bigl).
$$
The $B$-dependent terms are $\mathcal{L}^\theta_{X_1}i_{X_2}B - i_{X_2}d_{\theta} i_{X_1}B$.
Applying the identity
$$
  \mathcal{L}^\theta_X i_Y - i_Yd_{\theta} i_X = i_Yi_Xd_{\theta} + i_{[X,Y]}
$$
which follows from $\mathcal{L}^\theta_X=d_{\theta} i_X+i_Xd_{\theta}$ and~\eqref{cartan2}, we obtain
$$
  \mathcal{L}^\theta_{X_1}i_{X_2}B - i_{X_2}d_{\theta} i_{X_1}B
  = i_{X_2}i_{X_1}d_{\theta} B + i_{[X_1,X_2]}B.
$$
On the other hand, $e^{B}\bigl(e_{1}\circ_{\theta,H} e_{2}\bigr)$ contributes the term
$i_{[X_1,X_2]}B$ coming from the vector part $[X_1,X_2]$. Subtracting, we find that
the difference is precisely $(0,i_{X_2}i_{X_1}d_{\theta}B)$, which establishes~\eqref{eq:Bformula}.
The verification $d_{\theta}H' = d_{\theta}H -d_{\theta}^{2}B = 0$  confirms that
$(E,\circ_{\theta,H'})$ is a valid $\theta$-Courant algebroid, and that $e^{B}$ is the
required isomorphism.
\end{proof}

\begin{proposition}\label{prop:preserve}
Let $(\Lambda,H,\theta)$ be a $\theta$-almost twisted Poisson structure and let
$B\in\Omega^{2}(M)$. If the map
$$
\Phi_{B}=\mathrm{id}+B^\flat\circ\Lambda^\#
\colon T^*M\longrightarrow T^*M
$$
is invertible, then $\Phi_{B}$  induces  a morphism of $\mathcal{T}_{\mathrm{fix}}$.
\end{proposition}

\begin{proof}
By Proposition~\ref{prop:Bgauge}, $e^{B}$ is an isomorphism from $(E,\circ_{\theta,H})$
to $(E,\circ_{\theta,H'})$, where $H' = H - d_{\theta}B$. The image of $L_{\Lambda}$ under
$e^{B}$ is
$$
L' = e^{B}(L_{\Lambda})
= \Bigl\{\big(\Lambda^\#\xi\,,\;\xi+i_{\Lambda^\#\xi}B\big)\ \bigm|\ \xi\in T^*M\Bigr\}.
$$

Since $e^{B}$ intertwines the two $\theta$-Courant brackets, $L'$ is a Dirac structure for
$\circ_{\theta,H'}$.

Observe that $\xi+i_{\Lambda^\#\xi}B=\Phi_{B}(\xi)$. We set
$\alpha=\Phi_{B}(\xi)$.  The invertibility of $\Phi_{B}$ gives
$\xi=\Phi_{B}^{-1}(\alpha)$. Therefore,
\begin{equation*}
	L'=\Bigl\{\ \Big((\Lambda^\#\circ\Phi_{B}^{-1})\alpha\; ,\; \alpha\;\Big) \ \bigm|\ \alpha=\Phi_B(\xi),\ \xi\in T^*M \Bigr\}
	=\Bigl\{\ \Big((\Lambda^\#\circ\Phi_{B}^{-1})\alpha\; ,\; \alpha\;\Big)\ \bigm|\ \alpha\in T^*M\Bigr\}.
\end{equation*}
This proves that $L'$ is the graph of the bivector $\Lambda'$ defined by
\begin{align}\label{eq:ancre}
	(\Lambda')^\# &= \Lambda^\#\circ\Phi_{B}^{-1}.
\end{align}
By Proposition~\ref{dirac}, $(\Lambda',H',\theta)$ is a $\theta$-almost twisted
Poisson structure. Moreover, the condition $(\Lambda')^\#(\theta)=0$ follows from
$\Lambda^\#(\theta)=0$.

The restriction of $e^{B}$ to $L_{\Lambda}$ gives an isomorphism of vector bundles
$L_{\Lambda}\xrightarrow{\sim}L_{\Lambda'}$ that, by Proposition~\ref{prop:Bgauge},
respects the Lie algebroid brackets. Under the identifications
$T^*M\cong L_{\Lambda}$ and $T^*M\cong L_{\Lambda'}$, this isomorphism is
given by
$$
\Phi_{B}\colon T^*M\longrightarrow T^*M,\qquad
\Phi_{B}(\xi)=\xi+i_{\Lambda^\#\xi}B.
$$
 Since $\Phi_B$ is a smooth fiberwise linear isomorphism of $T^*M$,  $\Phi_B$ induces isomorphism of $C^\infty(M)$-module $\Phi_{B}\colon \Omega^{1}(M)\longrightarrow \Omega^{1}(M)$.
A direct verification using relation~\eqref{eq:ancre} yields
\begin{equation}\label{eq:phiB}
	\Phi_{B}\bigl([\xi,\eta]_{H,\theta}\bigr)
	=[\Phi_{B}(\xi),\Phi_{B}(\eta)]_{H',\theta},
	\qquad
	(\Lambda')^\#\circ\Phi_{B}=\Lambda^\#.
\end{equation}
Hence the proposition is proved.
\end{proof}

   An invertible  $\Phi_{B}$ is called  a gauge transformation; Its composition law  is described in  Proposition~\ref{prop:gauge}.

 \begin{proposition}
 	\label{prop:gauge}
 	Let $(\Lambda,H,\theta)\xrightarrow{\Phi_{B_1}}(\Lambda_{1},H_{1},\theta)
 	\xrightarrow{\Phi_{B_2,\Lambda_1}}(\Lambda_{2},H_{2},\theta)$ be two composable gauge transformations, where
 	$\Phi_{B_2,\Lambda_1}=\mathrm{id}+B_2^{\flat}\circ\Lambda_1^{\#}$ is constructed from the anchor $\Lambda_1^{\#}$ of the intermediate structure.
 	Then
 	$$
 	\Phi_{B_2,\Lambda_1}\circ\Phi_{B_1}=\Phi_{B_1+B_2},
 	$$
 	where the right-hand side is constructed using the anchor $\Lambda^{\#}$ of the original structure.
 	Consequently, $H_{2}=H-d_{\theta}(B_{1}+B_{2})$.
 \end{proposition}

\begin{proof}
	We set $\xi'=\Phi_{B_1}(\xi)=\xi+i_{\Lambda^{\#}\xi}B_1$. By the anchor condition satisfied by
	$\Phi_{B_1}$, we have $\Lambda_1^{\#}(\xi')=\Lambda^{\#}(\xi)$. Hence,
	\begin{align*}
		\Phi_{B_2,\Lambda_1}(\xi')
		=\xi' + i_{\Lambda_1^{\#}(\xi')}B_2 =\xi + i_{\Lambda^{\#}\xi}B_1 + i_{\Lambda^{\#}\xi}B_2 =\xi + i_{\Lambda^{\#}\xi}(B_1+B_2) =\Phi_{B_1+B_2}(\xi).
	\end{align*}
	The formula for $H_2$ follows from Proposition~\ref{prop:preserve} applied twice.
\end{proof}

\begin{remark}
	Proposition~\ref{prop:gauge} implies that for any gauge transformation $\Phi_B$, its inverse
	$\Phi_B^{-1}:=\Phi_{-B}$ is again a gauge transformation.
	Indeed, taking $B_1=B$ and $B_2=-B$ in Proposition~\ref{prop:gauge} gives
	$\Phi_{-B,\Lambda_B}\circ\Phi_B=\Phi_0=\mathrm{id}$.
	Therefore, gauge transformations do not form a subgroup of a fixed automorphism group
	in the naive sense (the formula for $\Phi_B$ depends on the object to which it is applied).
	Rather, the set
	$$
	\mathrm{Gau}(M):=\bigl\{\Phi_B:(\Lambda,H,\theta)\rightarrow(\Lambda_B,H-d_\theta B,\theta)\bigr\}
	$$
	where $\Phi_B=id+B^\flat\circ\Lambda^\#$ is invertible is a sub-groupoid of $\mathcal{T}_{\mathrm{fix}}$, on which composition corresponds to addition of $2$-forms.
\end{remark}

\begin{theorem}\label{thm:classification}
	Let $\Phi$ be a morphism of  $\mathcal{T}_{\mathrm{fix}}$ with source structure $(\Lambda,H,\theta)$.
	Then there exist $\Phi_B\in \mathrm{Gau}(M)$ and a $C^\infty(M)$-linear isomorphism
	$\Psi:\Omega^1(M)\to\Omega^1(M)$ such that
	$$
	\Phi=\Psi\circ\Phi_B.
	$$
	Moreover, $\Psi$ satisfies
	$$
	\Psi(\alpha)-\alpha \in \ker \Lambda^\# \quad \forall \alpha\in\Omega^1(M).
	$$
	In particular, $\Psi$ induces the identity on the quotient $\Omega^1(M)/\ker\Lambda^\#$.
\end{theorem}

\begin{proof}
	Let $(\Lambda,H,\theta)\xrightarrow{\Phi}(\Lambda',H',\theta)$ be a morphism of $\mathcal{T}_{\mathrm{fix}}$.
	Set $K=\ker\Lambda^\#$ and $D=\operatorname{im}\Lambda^\#\subset\mathfrak{X}^1(M)$. By Proposition~\ref{preservation_ancre},
	we have $\Phi(K)=K$.	
	We define a $C^\infty(M)$-bilinear form $T:\Omega^1(M)\times\Omega^1(M)\to C^\infty(M)$ by
	$
	T(\alpha,\beta)=\Lambda'\bigl(\Phi(\beta),\Phi(\alpha)\bigr)-\Lambda(\beta,\alpha).
	$
	This expression is antisymmetric and vanishes whenever at least one argument lies in $K$.
	Therefore, it descends to a well-defined form $\widetilde{T}$ on $D$.
	
	We extend $\widetilde{T}$ to a global $2$-form $B\in\Omega^2(M)$ as follows.
	For each open set $U\subset M$, choose a complement $E_U$ of $D|_U$ such that
	$\mathfrak{X}^1(U)=D|_U\oplus E_U$. We define a local $2$-form $B_U$ by setting
	$B_U(X,Y)=\widetilde{T}(X,Y)$ for $X,Y\in D|_U$, and $B_U=0$ if at least one argument lies in $E_U$.
	These local forms agree on intersections when restricted to $D\times D$.
	Hence, by using a partition of unity, we obtain a global $B$ such that
	$$
	B(\Lambda^\#\alpha,\Lambda^\#\beta)=T(\alpha,\beta)
	\quad \forall\,\alpha,\beta\in\Omega^1(M).
	$$
	
	Equivalently, we have
	\begin{align}\label{eq:preserve2}
		B(\Lambda^\#\alpha,Y)=\Big(\Phi(\alpha)-\alpha\Big)(Y)
		\quad \forall\,Y\in D.
	\end{align}
	
	We define $\Phi_B:\Omega^1(M)\to\Omega^1(M)$ by
	$\Phi_B(\alpha)=\alpha+i_{\Lambda^\#\alpha}B$. From \eqref{eq:preserve2} we compute
	\begin{align*}
		\Lambda^\#\bigl(\Phi_B\alpha-\Phi\alpha\bigr)(\beta)
		&= -\bigl(\Phi_B\alpha-\Phi\alpha\bigr)\bigl(\Lambda^\#\beta\bigr) \\
		&= -\alpha\bigl(\Lambda^\#\beta\bigr)-B\bigl(\Lambda^\#\alpha,\Lambda^\#\beta\bigr)+\Phi(\alpha)\bigl(\Lambda^\#\beta\bigr) \\
		&= -\alpha\bigl(\Lambda^\#\beta\bigr)-\bigl(\Phi(\alpha)-\alpha\bigr)\bigl(\Lambda^\#\beta\bigr)+\Phi(\alpha)\bigl(\Lambda^\#\beta\bigr) \\
		&= 0.
	\end{align*}
	Thus $\Phi_B\alpha-\Phi\alpha\in K$.  We claim that $\Phi_B$ is invertible.
Indeed, suppose $\Phi_B\alpha=0$. Since $\Phi_B\alpha-\Phi\alpha\in K$, we have $\Phi\alpha=-(\Phi_B\alpha-\Phi\alpha)\in K$.
  By Proposition~\ref{preservation_ancre}, $\alpha\in K$.  But for $\alpha\in K$,  $\Phi_B\alpha=\alpha$. So $\alpha=0$. Thus $\Phi_B$ is injective. Since $\Phi_B$ is a fiberwise endomorphism of the finite-rank vector bundle $T^*M$, the  injectivity of $\Phi_B$ implies its invertibility.
	Therefore, $\Phi_B$ is a gauge transformation from $(\Lambda,H,\theta)$ to
	$(\Lambda_B,H-d_\theta B,\theta)$, where $(\Lambda_B)^\#=\Lambda^\#\circ \Phi_B^{-1}$.
	
	We now set $\Psi=\Phi\circ\Phi_B^{-1}$. Then $\Psi$ is a morphism from $(\Lambda_B,H_B,\theta)$
	to $(\Lambda',H',\theta)$. Finally, putting $\beta=\Phi_B^{-1}(\alpha)$ yields
	\begin{align*}
		\Psi(\alpha)-\alpha
		= \Phi\bigl(\Phi_B^{-1}\alpha\bigr)-\alpha = \Phi(\beta)-\Phi_B(\beta) \in K,
	\end{align*}
	because $\Phi_B(\beta)-\Phi(\beta)\in K$. Hence $\Psi(\alpha)-\alpha\in K$ for all $\alpha$.
	The factorization $\Phi=\Psi\circ\Phi_B$ is thus established.
\end{proof}

 \begin{corollary}
 	Every morphism of  $\mathcal{T}_{\mathrm{fix}}$ with a non-degenerate source bivector $\Lambda$ is a gauge transformation.
 \end{corollary}

 \begin{proof} Let $\Phi$  be a morphism of  $\mathcal{T}_{\mathrm{fix}}$ with a non-degenerate source bivector $\Lambda$.
 	The map  $\Lambda^\#:\; \Omega^1(M) \rightarrow \mathfrak{X}^1(M)$ is an isomorphism. By Theorem~\ref{thm:classification},  a bivector field  $B$ is given as
 	$B(\Lambda^\# \xi, \Lambda^\# \eta) = (\Phi \xi - \xi)(\Lambda^\# \eta)$. So, we have $i_{\Lambda^\# \xi} B = \Phi \xi - \xi$. Consequently,
 	$\Phi_B(\xi) = \xi + i_{\Lambda^\# \xi} B = \Phi(\xi)$,
 	which shows that $\Phi_B = \Phi$. Hence every such a morphism is a gauge transformation.
 \end{proof}

\section{Isomorphisms shifting $\theta$: the family $\mathcal{T}_{\mathrm{mod}}$}
\label{sec:Tmod}
The complementary family $\mathcal{T}_{\mathrm{mod}}$ consists of morphisms with $\Delta(\Phi) = \theta' - \theta \neq 0$. 

Let $(\Lambda, H, \theta)$ be a $\theta$-atP structure on $M$, and let $f \in \mathrm{Cas}(\Lambda) := \{f \in C^\infty(M) : \Lambda^\#(df) = 0\}$. Define
\begin{equation}\label{eq:conformal-deformation}
	\Lambda_f = e^{-f}\Lambda, \qquad
	H_f = e^f H, \qquad
	\theta_f = \theta + df.
\end{equation}

\begin{proposition}\label{prop:struct}
	The triple $(\Lambda_f, H_f, \theta_f)$ defined by \eqref{eq:conformal-deformation} is a $\theta$-atP structure on $M$.
\end{proposition}

\begin{proof}
	We verify the four conditions of Definition~\ref{def_theta_alP} for $(\Lambda_f, H_f, \theta_f)$.
	
	The 1-form $\theta_f$ is closed since $d\theta = 0$. We have
	$$
	\Lambda_f^{\#}(\theta_f) = (e^{-f}\Lambda)^{\#}(\theta + df)
	= e^{-f}\left(\Lambda^{\#}(\theta) + \Lambda^{\#}(df)\right) = 0.
	$$
	
	For the Schouten-Nijenhuis bracket, the Casimir condition yields
	\begin{align*}
		\tfrac{1}{2}[\Lambda_f,\Lambda_f]
		= \tfrac{1}{2}[e^{-f}\Lambda, e^{-f}\Lambda] = \tfrac{1}{2}e^{-2f}[\Lambda, \Lambda] = \Lambda_f^{\#}(H_f).
	\end{align*}
	
	Finally,
	$$
	dH_f = d(e^f H) = e^f\,df \wedge H + e^f\,dH
	= e^f(df + \theta) \wedge H = \theta_f \wedge H_f.
	$$
\end{proof}

\begin{proposition}\label{thm:main}
	Let $f$ be a nonconstant Casimir function with respect to $\Lambda$. The $C^\infty(M)$-linear map
	$$
	\Psi_f : \Omega^1(M) \longrightarrow \Omega^1(M), \qquad \alpha \longmapsto e^f\alpha,
	$$
	is an isomorphism of $(H,\theta)$-twisted Lie algebroids from $A(\Lambda,H,\theta)$ to $A(\Lambda_f, H_f, \theta_f)$, hence a morphism  of $\mathcal{T}_{\mathrm{mod}}$ satisfying $\Delta(\Psi_f) = df$.
\end{proposition}

\begin{proof}
	The map $\Psi_f$ is  invertible, with inverse $\alpha\mapsto e^{-f}\alpha$. We verify the two conditions of \eqref{eq:morphism-bis}.
	
	\emph{Anchor condition:} We have
	$$
	\Lambda_f^\#(\Psi_f(\alpha)) = e^{-f}\Lambda^\#(e^f \alpha) = \Lambda^\#(\alpha).
	$$
	
	\emph{Bracket condition:} We compute
	\begin{align*}
		[\Psi_f(\alpha), \Psi_f(\beta)]_{H_f,\theta_f}
		&= [e^f\alpha, e^f\beta]_{H_f,\theta_f} \\
		&= e^f[\alpha,\beta]_{K} -e^f \Lambda(\alpha,\beta)df+ e^f i_{\Lambda^\#\beta}i_{\Lambda^\#\alpha}H + e^f \Lambda(\alpha,\beta)\theta_f \\
		&= e^f\left([\alpha,\beta]_{K} + i_{\Lambda^\#\beta}i_{\Lambda^\#\alpha}H + \Lambda(\alpha,\beta)\theta\right) \\
		&= \Psi_f([\alpha,\beta]_{H,\theta}).
	\end{align*}
	
	Finally, we have $\Delta(\Psi_f) = \theta_f - \theta = df$, which confirms that $\Psi_f \in \mathrm{Mor}(\mathcal{T}_{\mathrm{mod}})$.
\end{proof}

\begin{remark}
	When $f$ is locally constant, $df = 0$ and $\Psi_f$  is  a morphism of $\in \mathcal{T}_{\mathrm{fix}}$.
\end{remark}

We now give a description of the fiber $\Delta^{-1}(df)$ over an exact class.

Fix an object $X = (\Lambda, H, \theta)$ of $\mathcal{T}(M)$ and a Casimir function $f \in \mathrm{Cas}(\Lambda)$, and write $Y_f = (\Lambda_f, H_f, \theta_f)$ for the target of the conformal Casimir transformation $\Psi_f$ constructed in Proposition~\ref{prop:struct}. We wish to describe \emph{all} morphisms issuing from $X$ that realize the shift $df$ under the classifying functor. We define
$$
\mathcal{F}_X(df) \;:=\; \bigl\{ \Phi \in \mathrm{Mor}(\mathcal{T}(M)) \;:\; \mathrm{source}(\Phi) = X,\ \ \Delta(\Phi) = df \bigr\}.
$$

\begin{proposition}\label{prop:reduction}
	The map
	$$
	\tau \colon \mathcal{F}_X(df) \longrightarrow \bigl\{ \chi \in \mathcal{T}_{\mathrm{fix}} : \mathrm{source}(\chi) = Y_f \bigr\},
	\qquad \tau(\Phi) = \Phi \circ \Psi_f^{-1},
	$$
	is a bijection, with inverse $\sigma(\chi) = \chi \circ \Psi_f$.
\end{proposition}

\begin{proof}
	Let $\Phi \in \mathcal{F}_X(df)$, so that $\Phi \colon X \to (\Lambda', H', \theta')$ is a morphism of $\mathcal{T}(M)$ with $\Delta(\Phi) = \theta' - \theta = df$.
	
	The map $\Psi_f^{-1}$ has source $Y_f$ and target $X$. Then $\tau(\Phi) = \Phi \circ \Psi_f^{-1}$ is well composed. Moreover, by Property~\ref{prop:Delta-functeur}, we have
	$$
	\Delta\bigl(\Phi \circ \Psi_f^{-1}\bigr) \;=\; \Delta(\Phi) + \Delta\bigl(\Psi_f^{-1}\bigr)
	\;=\; \Delta(\Phi) - \Delta(\Psi_f) \;=\; df - df \;=\; 0.
	$$
	Hence $\tau(\Phi) = \Phi \circ \Psi_f^{-1}$ lies in $\mathcal{T}_{\mathrm{fix}}$. Thus $\tau$ and $\sigma$ are well defined.
	
	The associativity of composition in the groupoid $\mathcal{T}(M)$ gives
	$$
	\sigma(\tau(\Phi)) \;=\; \bigl(\Phi \circ \Psi_f^{-1}\bigr) \circ \Psi_f
	\;=\; \Phi \circ \bigl(\Psi_f^{-1} \circ \Psi_f\bigr) \;=\; \Phi \circ \mathrm{id} \;=\; \Phi.
	$$
	Symmetrically, for $\chi \in \mathcal{T}_{\mathrm{fix}}$ with source $Y_f$, we get
	$$
	\tau(\sigma(\chi)) \;=\; \bigl(\chi \circ \Psi_f\bigr) \circ \Psi_f^{-1}
	\;=\; \chi \circ \bigl(\Psi_f \circ \Psi_f^{-1}\bigr) \;=\; \chi \circ \mathrm{id} \;=\; \chi.
	$$
	This proves the proposition.
\end{proof}

Proposition~\ref{prop:reduction} shows that $\mathcal{F}_X(df)$ is entirely governed by the star of $\mathcal{T}_{\mathrm{fix}}$ at $Y_f$ (that is, by all morphisms of $\mathcal{T}_{\mathrm{fix}}$ issuing from $Y_f$), to which Theorem~\ref{thm:classification} applies directly, with source structure $Y_f = (\Lambda_f, H_f, \theta_f)$ in place of $(\Lambda, H, \theta)$. We now make the resulting description completely explicit.

\begin{lemma}\label{prop:conformal-gauge}
	Let $B \in \Omega^2(M)$ be such that $\Phi_B = \mathrm{id} + B^{\flat} \circ \Lambda_f^{\#}$ is invertible. Then the map
	$$
	\Theta_{f,B} \;:\; \Omega^1(M) \longrightarrow \Omega^1(M), \qquad \alpha \longmapsto e^f \alpha + i_{\Lambda^{\#}\alpha} B
	$$
	is a morphism of $\mathcal{T}(M)$ with $\Delta(\Theta_{f,B}) = df$.
\end{lemma}

\begin{proof}
	We claim that the composite $\Phi_B \circ \Psi_f$ is well defined in $\mathcal{T}(M)$. Indeed, according to Proposition~\ref{thm:main}, $\Psi_f$ is a morphism from $X = (\Lambda, H, \theta)$ to $Y_f = (\Lambda_f, H_f, \theta_f)$. By Proposition~\ref{prop:preserve}, since $\Phi_B = \mathrm{id} + B^{\flat} \circ \Lambda_f^{\#}$ is assumed invertible, $\Phi_B$ is an element of $\mathcal{T}_{\mathrm{fix}}$ with source $Y_f$, namely a morphism from $Y_f$ to $(\Lambda_f^B, H_f - d_{\theta_f} B, \theta_f)$, where $(\Lambda_f^B)^{\#} = \Lambda_f^{\#} \circ \Phi_B^{-1}$. As $\mathcal{T}(M)$ is a groupoid, the composite $\Phi_B \circ \Psi_f$ is again a morphism of $\mathcal{T}(M)$, with source $X$ and target $(\Lambda_f^B, H_f - d_{\theta_f} B, \theta_f)$, as claimed.
	
	Moreover, we have
	$$
	\Delta(\Phi_B \circ \Psi_f) = \theta_f - \theta = (\theta + df) - \theta = df.
	$$
	
	We now derive the explicit formula for $\Phi_B \circ \Psi_f$. Using the explicit formulas for $\Psi_f$ and $\Phi_B$, we obtain
	$$
	\Phi_B \circ \Psi_f(\alpha) = \Phi_B(e^f \alpha) = e^f \alpha + i_{\Lambda_f^{\#}(e^f \alpha)} B = e^f \alpha + i_{\Lambda^{\#}\alpha} B = \Theta_{f,B}(\alpha).
	$$
	The lemma is thus proved.
\end{proof}

\begin{lemma}\label{lem:B-equiv}
	Let $B_1, B_2 \in \Omega^2(M)$ be such that $\Theta_{f,B_1}$ and $\Theta_{f,B_2}$ are both defined as in Lemma~\ref{prop:conformal-gauge}. Then
	$\Theta_{f,B_1} = \Theta_{f,B_2}$ if and only if
	$B_1 - B_2 \in \mathrm{Ann}(D) := \bigl\{\, C \in \Omega^2(M) \;:\; i_X C = 0 \ \ \forall X \in D \,\bigr\}.$
\end{lemma}

\begin{proof}
	Let $\alpha \in \Omega^1(M)$. We have
	$$
	\Theta_{f,B_1}(\alpha) - \Theta_{f,B_2}(\alpha)
	= \bigl(e^f \alpha + i_{\Lambda^{\#}\alpha} B_1\bigr) - \bigl(e^f \alpha + i_{\Lambda^{\#}\alpha} B_2\bigr)
	= i_{\Lambda^{\#}\alpha}(B_1 - B_2).
	$$
	Hence $\Theta_{f,B_1} = \Theta_{f,B_2}$ as maps on $\Omega^1(M)$ if and only if
	$i_{\Lambda^{\#}\alpha}(B_1 - B_2) = 0$ for every $\alpha \in \Omega^1(M)$, that is, if and only if $B_1 - B_2$
	vanishes identically on $D \times \mathfrak{X}^1(M)$. This is precisely
	the condition $B_1 - B_2 \in \mathrm{Ann}(D)$, which proves the lemma.
\end{proof}

\begin{theorem}\label{thm:structure-fibre}\mbox{ }\vspace*{-.2cm}
	\begin{enumerate}[label=(\roman*)]
		\item Every $\Phi \in \mathcal{F}_X(df)$ factors as $\Phi = \Psi \circ \Theta_{f,B}$ for some $B \in \Omega^2(M)$
		rendering $\Theta_{f,B}$ invertible, and some $C^{\infty}(M)$-linear automorphism $\Psi$ of $\Omega^1(M)$
		satisfying $\Psi(\alpha) - \alpha \in K$ for every $\alpha \in \Omega^1(M)$. In particular, $\Psi$ induces the
		identity on $\Omega^1(M)/K$;
		\item The assignment $B \mapsto \Theta_{f,B}$ induces an injection $\Omega^2(M)/\mathrm{Ann}(D) \hookrightarrow \mathcal{F}_X(df)$.
	\end{enumerate}
\end{theorem}

\begin{proof}
(i) Let $\Phi\in\mathcal F_X(df)$. By Proposition~\ref{prop:reduction}, $\chi=\tau(\Phi)=\Phi\circ\Psi_f^{-1}$
is the unique element of $\mathcal{T}_{\mathrm{fix}}$ with source $Y_f$ such that $\Phi=\chi\circ\Psi_f$; this
identity, and its uniqueness, follow from $\tau$ being a bijection with inverse
$\sigma(\chi)=\chi\circ\Psi_f$. By Theorem~\ref{thm:classification}, There exists  a gauge
transformation $\Phi_B\in Gau(M)$ based at $Y_f$, together with a $ C^{\infty}(M)$-linear automorphism $\Psi$ of
$\Omega^{1}(M)$  such that
$$
\chi = \Psi\circ\Phi_B, \qquad \Psi(\alpha)-\alpha \in \ker\Lambda_f^{\#} \quad \text{for every } \alpha\in\Omega^{1}(M).
$$
Since $\Lambda_f^{\#}=e^{-f}\Lambda^{\#}$, we have $\ker\Lambda_f^{\#}=
\ker\Lambda^{\#}=K$. Hence
$\Psi(\alpha)-\alpha\in K$ for every $\alpha\in\Omega^{1}$, as required.

Moreover, by Lemma~\ref{prop:conformal-gauge},  one gets
$$
\Phi \;=\; (\Psi\circ\Phi_B)\circ\Psi_f \;=\; \Psi\circ(\Phi_B\circ\Psi_f) \;=\; \Psi\circ\Theta_{f,B},
$$
which proves (i).

(ii) The assignment $B\mapsto\Theta_{f,B}$, defined on the set of $2$-forms $B$ for which
$\Theta_{f,B}$ is invertible, takes values in $\mathcal F_X(df)$ by Lemma~\ref{prop:conformal-gauge}.
Lemma~\ref{lem:B-equiv} shows precisely that two such $2$-forms $B_1,B_2$ have the same image under
this assignment if and only if $B_1-B_2\in Ann(D)$; equivalently, the induced map on the quotient
$\Omega^{2}(M)/ Ann(D)$ is injective. This proves (ii) and completes the proof of the theorem.
\end{proof}
The following is a direct consequence of Theorem~\ref{thm:structure-fibre}.

\begin{corollary}\label{cor:nondeg}
	If $\Lambda$ is nondegenerate, then necessarily $\Psi = \mathrm{id}$
	in (i) of Theorem~\ref{thm:structure-fibre}, and the assignment
	$$
	B \longmapsto \Theta_{f,B}
	$$
	is a bijection between $\{\, B \in \Omega^2(M) : \Theta_{f,B} \text{ is invertible} \,\}$ and $\mathcal{F}_X(df)$.
\end{corollary}

\begin{remark}\label{rem:fiber-remark}
	Propositions~\ref{prop:reduction} together with Lemma~\ref{prop:conformal-gauge}, and
	Theorem~\ref{thm:structure-fibre} show that the fiber $\Delta^{-1}(df)$, for $f \in \mathrm{Cas}(\Lambda)$, is governed
	by exactly the same data as the fiber $\Delta^{-1}(0) = Mor(\mathcal{T}_{\mathrm{fix}})$ itself, namely a $2$-form $B$
	modulo $\mathrm{Ann}(D)$, together with an automorphism $\Psi$ trivial on $\Omega^1(M)/K$, merely translated by
	the fixed reference point $\Psi_f$. The Casimir function $f$ thus records which exact class
	is realized, while $B$, and $\Psi$ (when $K \neq 0$) parametrize the fiber above it.
\end{remark}

We now consider a closed $1$-form $\eta \in Z^1_{dR}(M)$ with $\Lambda^{\#}\eta = 0$, and we ask whether $\eta$ itself or, more generally, any class
represented by such a $1$-form can arise as $\Delta(\Phi)$ for some $\Phi \in \mathrm{Mor}(\mathcal{T}(M))$. When $[\eta] \neq 0$, no global potential exists, and the analysis above no longer applies. If $\eta \in \Omega^1(M)$ satisfies $\Lambda^{\#}\eta = 0$ and $\eta_x \neq 0$ at some point $x \in M$, then
$K_x = \ker \Lambda^{\#}_x \neq 0$. In particular, if $\Lambda$ is nondegenerate on $M$,
any $1$-form $\eta$ with $\Lambda^{\#}\eta = 0$ is identically zero, and the question of
non-exact fibers is vacuous. By Proposition~\ref{preservation_ancre}, every morphism of $\mathcal{T}(M)$ preserves the
characteristic sub-bundle $K$. Thus non-exact shifts can therefore only occur when
the source bivector field $\Lambda$ fails to be non-degenerate.

\begin{definition}\label{def:Tcg}
	We call the conformal-gauge sub-groupoid of $\mathcal{T}(M)$ the sub-groupoid generated, under composition and inversion, by $\mathrm{Gau}(M)$
	and by all conformal Casimir transformations $\Psi_f$, $f \in \mathrm{Cas}(\Lambda)$, taken over
	every object $(\Lambda, H, \theta)$ of $\mathcal{T}(M)$. We denote it by $\mathcal{T}_{\mathrm{cg}}$.
\end{definition}
We  now suppose that $M$ is a manifold with non-zero first
de~Rham cohomology, that is \ $H_{\mathrm{dR}}^{1}(M)\neq 0$.
\begin{proposition}\label{prop:nogo}
	Let $\eta$ be a closed $1$-form with
	$[\eta] \neq 0 \in H^1_{dR}(M, \mathbb{R})$,
	$$
	\Delta^{-1}(\eta) \cap \mathrm{Mor}(\mathcal{T}_{\mathrm{cg}}) = \varnothing.
	$$
\end{proposition}

\begin{proof}
	It suffices to show that the image under $\Delta$ of every morphism belonging to $\mathcal{T}_{\mathrm{cg}}$ lies in the
	subgroup $B^1_{dR}(M)$ of exact closed $1$-forms.
	
	We compute  $\Delta$ on each type of generator. If
	$\Phi_B \in \mathrm{Gau}(M)$ is a gauge transformation, then $\Delta(\Phi_B) = 0 \in B^1_{dR}(M)$. If $\Psi_f$ is a conformal Casimir transformation, then
	$\Delta(\Psi_f) = df$, which is exact.
	
	Let $\chi \in \mathrm{Mor}(\mathcal{T}_{\mathrm{cg}})$. By Definition~\ref{def:Tcg}, $\chi$ is obtained from $n \in \mathbb{N}^*$
	composable words.
	$$
	\chi = \chi_n \circ \chi_{n-1} \circ \cdots \circ \chi_1,
	$$
	where each $\chi_k$ is either a gauge transformation, a conformal Casimir transformation, or the
	inverse of one of these (based at the appropriate intermediate object of $\mathcal{T}(M)$ so that the
	composition is defined at each stage). The additivity of $\Delta$ under
	composition gives
	$$
	\Delta(\chi) = \Delta(\chi_n) + \Delta(\chi_{n-1}) + \cdots + \Delta(\chi_1),
	$$
	where $\Delta(\chi_k)\in\{0,df,-df\}\subset B^1_{dR}(M)$. Since
	$B^1_{dR}(M)$ is a subgroup of $\bigl(Z^1_{dR}(M), +\bigr)$,  $\Delta(\chi) \in B^1_{dR}(M)$.
	Therefore  $\Delta(\mathrm{Mor}(\mathcal{T}_{\mathrm{cg}})) \subseteq B^1_{dR}(M)$. Hence $\Delta^{-1}(\eta) \cap \mathrm{Mor}(\mathcal{T}_{\mathrm{cg}}) = \varnothing$.
\end{proof}

Proposition~\ref{prop:nogo} shows that the conformal-gauge
sub-groupoid $\mathcal{T}_{\mathrm{cg}}$ cannot produce morphisms whose $\Delta$-image is a non-exact closed $1$-form. The full groupoid $\mathcal{T}(M)$ can, however have 
non-empty fibers above non-exact classes: fixe a closed  non-exact $1$-form $\eta$. Consider the two $\theta$-atP structures
$(\Lambda, H, \theta) = (0, 0, 0)$ and $(\Lambda', H', \theta') = (0, 0, \eta)$ on $M$. The identity map $\Phi = \mathrm{id}: \Omega^1(M) \longrightarrow \Omega^1(M)$ is
trivially a morphism of $\mathcal{T}(M)$ from $(0, 0, 0)$ to $(0, 0, \eta)$ with $\Delta(\Phi) = \theta' - \theta = \eta$. The  description of 
$\Delta^{-1}(\eta)$ for non-exact $\eta$ remains  open.

\end{document}